\documentclass[10pt,envcountsect]{svjour3}

\smartqed
\RequirePackage{fix-cm}
\usepackage{graphicx}
\usepackage{psfrag}
\usepackage{subfigure}
\usepackage{url}
\usepackage{color}
\usepackage{cite}
\usepackage{epsfig}
\usepackage{amsmath,amssymb}
\usepackage{array,booktabs}
\usepackage{graphics,epsfig}
\usepackage{grffile}
\usepackage{amsmath}
\usepackage{color}
\usepackage{verbatim}
\usepackage{amsfonts}
\usepackage{epsfig}
\usepackage{float}

\newcommand{\rset}{\mathbb{R}}

\newcommand{\bmtrx}{\left[\begin{array}}
\newcommand{\emtrx}{\end{array}\right]}

\newtheorem{assumption}[theorem]{Assumption}

\usecounter{algorithmenumi}

\begin{document}

\title{ Random block coordinate descent methods for linearly constrained
optimization over networks}

\author{ I. Necoara,  Yu. Nesterov and F. Glineur}

\institute{I. Necoara  is with Automatic Control and Systems
Engineering Department, University Politehnica Bucharest, 060042
Bucharest, Romania. \email{\texttt{ion.necoara@acse.pub.ro}}. \\
Yu. Nesterov and F. Glineur are with Center for Operations Research
and Econometrics, Catholic University of Louvain, Louvain-la-Neuve,
B-1348, Belgium. \email{ \texttt{\{Yurii.Nesterov,
Francois.Glineur\}@uclouvain.be}}. \\
\textit{The present paper is a generalization of our paper:  I.
Necoara, Yu. Nesterov  and  F. Glineur, A random coordinate descent
method on large-scale optimization problems with linear constraints,
Tech. Rep., Univ. Politehnica Bucharest, 1-25, July 2011 (ICCOPT
2013, Lisbon).} }

\titlerunning{Random  coordinate descent methods on large optimization problems}

\date{Received: March 2013 / Updated: November 2014}

\maketitle

\begin{abstract}
In this paper we develop  random block  coordinate gradient descent
methods for minimizing large scale linearly constrained separable
convex problems over networks.  Since we have coupled constraints in
the problem, we  devise an algorithm that updates in parallel $\tau
\geq 2$ (block) components per iteration.  Moreover, for this method
the computations can be performed in a distributed fashion according
to the structure of the network.   However, its complexity per
iteration  is usually  cheaper than of the full gradient method when
the number of nodes $N$ in the network is large. We prove that for
this method  we obtain in expectation an $\epsilon$-accurate
solution in at most $\mathcal{O}(\frac{N}{\tau \epsilon})$
iterations and thus the convergence rate depends linearly on the
number of (block) components $\tau$ to be updated. For strongly
convex functions  the new method converges linearly. We also focus
on how to choose the probabilities to make the randomized algorithm
to converge as fast as possible and we arrive at solving a sparse
SDP.  Finally, we describe several applications that fit in our
framework, in particular the convex feasibility problem.
Numerically, we show that the parallel  coordinate  descent method
with $\tau>2$  accelerates on its basic counterpart  corresponding
to $\tau=2$.
\end{abstract}



\section{Introduction}
\label{s_intro}

\noindent The performance  of a network  composed  of interconnected
subsystems can be improved if the traditionally separated subsystems
are optimized together. Recently, coordinate descent methods have
emerged as a powerful tool for solving large  data network problems:
e.g. resource allocation \cite{Nec:13,XiaBoy:06}, coordination in
multi-agent systems \cite{IshTem:12,Nec:13,YouXie:11}, estimation in
sensor networks or distributed control \cite{NecCli:13}, image
processing \cite{BauBor:96,Com:96,Wri:10} and other areas
\cite{LiuWri:14,QinSch:10,RicTac:13}.  The problems  we consider in
this paper have the following features: {\it the size of data}  is
big so that usual methods based on whole gradient computations are
prohibitive. Moreover {\it the incomplete structure of information}
(e.g.  the data are distributed over the  nodes of the network, so
that at a given time we need to work only with the data available
then)  may also be an obstacle for whole gradient computations. In
this case, an appropriate way to approach these problems is through
coordinate descent methods. These methods were among the first
optimization methods studied in literature  but until recently they
haven't received much attention.

\noindent The main differences in all variants of coordinate descent
methods consist in the criterion of choosing at each iteration the
coordinate over which we minimize the  objective function and the
complexity of this choice. Two classical criteria used often in
these algorithms are the cyclic  and the greedy  coordinate descent
search, which significantly differs by the amount of computations
required to choose the appropriate index. For cyclic coordinate
search estimates on the rate of convergence were given recently in
\cite{BecTet:12}, while for the greedy coordinate search (e.g.
Gauss-Southwell rule) the convergence rate is given e.g. in
\cite{Tse09}.  One paper related to our work is \cite{Bec:12}, where
a 2-coordinate greedy descent method is developed for minimizing a
smooth function subject to a single linear equality constraint and
additional bound constraints on the decision variables.  Another
interesting approach is based on random choice rule, where the
coordinate search is random. Recent complexity results on random
coordinate descent methods  for smooth convex objective functions
were obtained in \cite{Nes:10,Nec:13}. The extension to composite
convex objective functions was given e.g. in
\cite{NecPat:14,RicTac:11}. These methods are inherently serial.
Recently, parallel and distributed implementations of coordinate
descent methods were also analyzed e.g. in
\cite{LiuWri:14,NecCli:13,NecCli:14,RicTac:13}.

\noindent \textit{Contributions}:   In this paper we develop random
block coordinate gradient  descent methods suited for large
optimization problems in networks where the information cannot be
gather centrally, but rather it is distributed over the network.
Moreover, in our paper we focus on optimization problems with
linearly coupled constraints (i.e. the constraint set is coupled).
Due to the coupling in the constraints we  introduce  a $\tau \geq
2$ block variant of random coordinate gradient  descent method, that
involves at each iteration the closed form solution of an
optimization problem only with respect to $\tau$ block variables
while keeping all the other variables fixed. Our approach allows us
to analyze in the same framework several methods: full gradient,
serial random coordinate descent and any parallel random coordinate
descent method in between. For this method we obtain for the
expected values of the objective function a convergence rate
$\mathcal{O}(\frac{N}{\tau k})$, where $k$ is the iteration counter
and $N$ is the number of nodes in the network. Thus, the theoretical
speedup in terms of the number of iterations needed to approximately
solve the problem, as compared to the basic method corresponding to
$\tau=2$,  is an expression depending on the number of  components
$\tau$ to be updated (processors)   and for a complete network  the
speedup is equal to $\tau$ (number of components updated). This
result also  shows that the speedup achieved by our method  on the
class of separable problems with coupling constraints is the same as
for separable problems without coupling constraints. For strongly
convex functions we prove that the new method converges linearly. We
also focus on how to choose the probabilities to make the randomized
algorithm to converge as fast as possible and we arrive at solving
sparse  SDPs. While the most obvious benefit of randomization is
that it can lead to faster algorithms, either in worst case
complexity analysis and/or numerical implementation, there are also
other benefits  of our algorithm that are at least as important:
e.g., the use of randomization leads to a simpler algorithm that is
easier to analyze,   produces a more robust output and  can often be
organized to exploit modern computational architectures (e.g
distributed and parallel~computers).

\noindent \textit{Contents}:  The paper is organized as follows. In
Section \ref{s_formulation} we introduce our optimization model and
assumptions. In Section \ref{sec_nonsep} we propose a  random block
coordinate descent algorithm and derive the convergence rate in
expectation. Section \ref{sec_opt_prob} provides means to choose
optimally the probability distribution. In Section  \ref{sec_app} we
discuss possible applications and we conclude with some preliminary
numerical results in Section \ref{sec_num_simulations}.

\noindent \textit{Notation}: We work in the space $\rset^N$ composed
by column vectors.  For $x, y \in \rset^N$ denote the standard
Euclidian inner product $\langle x, y \rangle = x^Ty$ and the
Euclidian norm $\|x\| = \langle x, x \rangle^{1/2} $. For symmetric
matrices $X,Y$ we consider the inner product $\langle X, Y \rangle =
\text{trace}(XY)$.  We use the same notation $\langle \cdot, \cdot
\rangle$ and $\| \cdot \|$ for spaces of different dimension. We
define the partition of the identity matrix: $I_{N} = [e_1 \cdots
e_N]$, where $e_i \in \rset^{N}$. Then, for any $x \in \rset^{N}$ we
write $x = \sum_i e_i x_i$. Moreover, $D_x$ denotes the diagonal
matrix with the entries $x$ on the diagonal and $x^{-1} =[x_1^{-1}
\cdots x_N^{-1}]^T$. We denote with $e \in \rset^N$ the
$N-$dimensional vector with all entries equal to one. For a positive
semidefinite matrix $W \in \rset^{N \times N}$ we consider the
following order on its eigenvalues $0 \leq \lambda_1 \leq \cdots
\leq \lambda_N$ and $\|x\|_W^2 = x^T W x$.


\section{Problem formulation}
\label{s_formulation} \noindent We consider large data network
optimization problems where each agent in the network is associated
with a local variable so that their sum is fixed and we need to
minimize a separable convex objective function:
\begin{equation}
\label{scpp}
\begin{aligned}
f^*= & \min\limits_{x_i \in \rset}   f_1(x_1)+ \cdots + f_N(x_N)  \\
     & \text{s.t.:} \;\; x_1 + \cdots + x_N = 0.
\end{aligned}
\end{equation}
For convenience, we will focus on scalar convex functions   $f_i :
\rset \to \rset$, i.e. $x_i \in \rset$, in the optimization model
\eqref{scpp}. However, our results can be easily extended  to the
block case, when $x_i \in \rset^n$, with $n \geq 1$, using the
Kronecker product. Moreover, constraints of the form $\alpha_1 x_1 +
\cdots + \alpha_N x_N = b$, where $x_i \in \rset^n$ and $\alpha_i
\in \rset$, can be easily handled in our framework by a change of
coordinates.

\noindent Optimization problems with linearly coupled constraints
\eqref{scpp} arise in many areas such as resource allocation
\cite{Nec:13,XiaBoy:06}, coordination in multi-agent systems
\cite{IshTem:12,Nec:13,YouXie:11}, image processing
\cite{BauBor:96,Com:96,NecCli:13,Wri:10} and other areas
\cite{LiuWri:14,QinSch:10,RicTac:13}.  For  problem \eqref{scpp} we
associate a network composed of several nodes $[N]= \{1, \cdots,
N\}$ that can exchange information according to a communication
graph ${\cal G}=([N],E)$, where $E$ denotes the set of edges, i.e.
$(i, j) \in E \subseteq [N] \times [N]$ models that node $j$ sends
information to node $i$. We assume that the graph ${\cal G}$ is
undirected and connected. For an integer $\tau \geq 2$, we also
define with ${\cal P}_\tau$ the set of paths of $\tau$ vertices in
the graph. Note that we have at most $\binom{\tau}{N}$ paths of
$\tau$ vertices in a graph. The local information structure imposed
by the graph ${\cal G}$ should be considered as part of the problem
formulation.

\noindent  Our goal  is to devise a distributed algorithm that
iteratively solves the convex problem \eqref{scpp} by passing the
estimate of the optimizer only between neighboring nodes along paths
of $\tau$ vertices. There is great interest in designing such
distributed and parallel algorithms, since centralized algorithms
scale poorly with the number  of nodes and are less resilient  to
failure of the central node.  We  use the notation:
\begin{align*}
x & = [x_1 \cdots x_N]^T \quad \text{and} \quad f(x)  =f_1(x_1)+ \cdots + f_N(x_N).
\end{align*}

\noindent Let us define the extended subspace $S \subseteq \rset^{N}$ and its orthogonal
complement $T \subseteq \rset^{N}$:
\[ S= \left \{ x: \; \sum_{i=1}^N x_i =0 \right\}, \qquad
T= \{u: \;  u_1 = \cdots = u_N \}.  \]

\noindent The  basic assumption considered in this paper is:
\begin{assumption}
\label{ass} We assume that each function $f_i$ is convex and has
Lipschitz continuous gradient with constants $L_i>0$,
i.e. the following inequality holds:
\begin{equation}
\label{grdLp}
 \| \nabla f_i(x_i) - \nabla f_i(y_i) \| \leq L_i \|x_i - y_i
\| \quad  \forall x_i, y_i \in \rset.
\end{equation}
\end{assumption}

\noindent From the Lipschitz property of the gradient \eqref{grdLp},
the following inequality holds  for all $x_i, d_i \in
\rset$ \cite{Nes:04}:
\begin{align}
\label{lip_in} f_i(x_i+d_i) \leq f_i(x_i) + \langle \nabla f_i(x_i),
d_i \rangle + \frac{L_i}{2} \|d_i\|^2.
\end{align}

\noindent We  denote with $X^*$ the set of optimal solutions for
problem \eqref{scpp}. Note that $x^*$ is optimal solution for \eqref{scpp} if and only if:
\[ \sum_{i=1}^N  x_i^* = 0, \;\;\; \nabla f_i (x_i^*) = \nabla f_j (x_j^*) \quad \forall i \not = j \in [N]. \]


\section{Random  coordinate descent algorithms}
\label{sec_nonsep} \noindent In this section we devise randomized
block coordinate gradient  descent algorithms for solving the
separable convex problem \eqref{scpp} and analyze their convergence.
Since we have coupled constraints in the problem, the algorithm has
to update in parallel $\tau \geq 2$ components per iteration.
Usually, the algorithm  can be accelerated by parallelization, i.e.
by using more than one pair of coordinates per iteration.  Our
approach allows us to analyze in the same framework several methods:
full gradient ($\tau=N$), serial random coordinate descent
($\tau=2$) and any parallel random coordinate descent method in
between ($<2 \tau < N$).  Let us fix $N \geq \tau \geq 2$ and we
denote with ${\cal N} \in {\cal P}_\tau$  a path of $\tau$ vertices
in the connected undirected  graph ${\cal G}$. We also assume
available a probability distribution $p_{\cal N}$ over the set
${\cal P}_\tau$ of paths  of $\tau$ vertices in the  graph ${\cal
G}$. Then, we can derive a randomized $\tau$  coordinate descent
algorithm where we update at each iteration only $\tau$ coordinates
in the vector $x$. Let us define ${\cal N} = (i_1, \cdots i_\tau)
\in {\cal P}_\tau$, with $i_l \in [N]$, $s_{\cal N} = [s_{i_1}
\cdots s_{i_\tau}]^T \in \rset^{\tau}$, $L_{\cal N} = [L_{i_1}
\cdots L_{i_\tau}]^T \in \rset^{\tau}$ and $\nabla f_{\cal N}
=[\nabla f_{i_1} \cdots \nabla f_{i_\tau}]^T \in \rset^{\tau}$.
Under assumption \eqref{grdLp} the following inequality holds:
\begin{equation}
\label{lipN}
 f(x+ \sum_{i \in {\cal N}} e_i s_i) \le f(x) + \langle \nabla
 f_{\cal N}(x), s_{\cal N} \rangle + \frac{1}{2}
 \| s_{\cal N}  \|^2_{D_{L_{\cal N}}}.
\end{equation}

\noindent Based on the inequality \eqref{lipN}  we can devise a
general randomized $\tau$  coordinate descent algorithm for
problem \eqref{scpp}, let us call it $RCD_\tau$. Given an $x$ in the
feasible set $S$, we choose the coordinate $\tau$-tuple ${\cal N}
\in {\cal P}_\tau$ with probability $p_{\cal N}$. Let the next
iterate be chosen as follows:
\begin{equation*}
x^+ = x + \sum_{i \in {\cal N}} e_i d_i,
\end{equation*}
i.e. we update $\tau$  components in the vector $x$, where
the direction $d_{\cal N}$ is  determined by requiring that the next
iterate $x^+$ to be also feasible for \eqref{scpp} and minimizing
the right hand side in \eqref{lipN}, i.e.:
\[ d_{\cal N} = \arg \min_{s_{\cal N}: \sum_{i \in {\cal N}} s_{i} =0} f(x) +
\langle \nabla f_{\cal N}(x), s_{\cal N} \rangle + \frac{1}{2}
 \| s_{\cal N}  \|^2_{D_{L_{\cal N}}} \]
or explicitly, in closed form:
\[  d_i =  \frac{1}{L_i} \frac{\sum_{j \in {\cal N}} \frac{1}{L_j} \left( \nabla f_j(x_j) - \nabla f_i(x_i) \right)} {\sum_{j \in {\cal N}} \frac{1}{L_j}}  \qquad \forall i \in {\cal N}. \]

\noindent In conclusion, we obtain the following randomized $\tau$
coordinate gradient descent method:
\begin{center}
\fbox{
\begin{minipage}[c]{8.5cm}
\textbf{Algorithm $RCD_\tau$}
\begin{equation*}
\begin{split}
&1. \; \text{ choose $\tau$-tuple ${\cal
N}_k = (i_1^k, \cdots i^k_\tau)$ with probability $p_{\cal N}$}    \\
&2. \;  \text{set $x^{k+1}_i = x^k_i \;\; \forall i \notin {\cal N}_k$
and $x^{k+1}_{i} = x^k_{i} + d^k_i  \;\; \forall i \in {\cal N}_k$.} \\
\end{split}
\end{equation*}
\end{minipage}
}
\end{center}

\noindent  Clearly, algorithm $RCD_\tau$ is distributed since only
neighboring nodes along a path in the graph need to communicate at
each iteration. Further, at each iteration only $\tau$ components of
$x$ are updated, so that our method has low complexity per
iteration.  Finally, in our algorithm we maintain feasibility at
each iteration, i.e. $x_{1}^k + \cdots + x_N^k   = 0$ for all $k
\geq 0$. The random choice of coordinates makes the algorithm
adequate for parallel and distributed implementations  and thus more
flexible than greedy coordinate descent methods \cite{Bec:12,Tse09}.
In particular, for $\tau=2$, we choose one pair of connected nodes
$(i_k,j_k) \in E$ with some given  probability $p_{i_kj_k}$ and
obtain the following basic iteration:
\[ x^{k+1} = x^k + \frac{1}{L_{i_k} + L_{j_k}} (e_{i_k} - e_{j_k})
\left( \nabla f_{j_k}(x_{j_k}^k ) - \nabla f_{i_k}(x^k_{i_k})
\right). \]
Note that our algorithm belongs to the class of
center-free methods (in \cite{XiaBoy:06} the term center-free refers
to the absence of a coordinator) with the following iteration:
\begin{align}
\label{center_free} x^{k+1}_i = x^k_i + \sum_{j \in {\cal N}}
w_{ij}^k \left( \nabla f_j(x^k_j) - \nabla f_i(x^k_i) \right) \qquad
\forall i \in [N],
\end{align}
with appropriate weights $w_{ij}^k$. Based on the inequality
\eqref{lipN} and the optimality conditions for the subproblem corresponding to $d_{\cal N}$,  the following decrease in the objective function values
can be derived:
\begin{align*}
f(x^+) & \leq   f(x) - \sum_{i \in {\cal N}} \frac{
\left| \sum_{j \in {\cal N}} \frac{1}{L_j} \left( \nabla f_j(x_j) -
\nabla f_i(x_i) \right) \right|^2} {2L_i(\sum_{j \in {\cal N}} \frac{1}{L_j})^2} \\
& = f(x) - \frac{1}{2} \nabla f(x)^T G_{{\cal N}} \nabla f(x),
\end{align*}
where the matrix $G_{{\cal N}}$  is defined as follows:
\begin{align}
\label{exp_GN}
G_{{\cal N}}=  D_{L_{\cal N}}^{-1} - \frac{1}{\sum_{i \in {\cal N}} 1/L_i}
L^{-1}_{\cal N} (L^{-1}_{\cal N})^T,
\end{align}
where, with an abuse of notation,  $L_{\cal N} \in \rset^N$  denotes
the vector with components zero outside the index set ${\cal N}$ and
components $L_i$ for $i \in {\cal N}$.   Therefore, taking the
expectation over the random $\tau$-tuple ${\cal N} \in {\cal
P}_\tau$, we obtain the following inequality:
\begin{equation}
\label{expdecreaseM} E [ f(x^+) \; | \;x ] \le f(x) - \frac{1}{2}
\nabla f(x)^T G_\tau \nabla f(x),
\end{equation}
where $G_\tau = \sum_{\cal N \in {\cal P}_\tau} p_{\cal N} G_{\cal
N}$ and can be interpreted as a weighted Laplacian for the graph
${\cal G}$. From the decrease in the objective function values given
above, it follows immediately that the matrix $G_\tau$ is positive
semidefinite and has an eigenvalue $\lambda_1(G_\tau) = 0$ with the
corresponding eigenvector  $e \in T$. Since the graph is connected, it
also follows that the eigenvalue $\lambda_1(G_\tau) = 0$ is simple,
i.e. $\lambda_2(G_\tau) > 0$.

\noindent On the extended subspace $S$ we now define a norm  that
will be used subsequently for measuring distances in this subspace.
We define the primal ``norm'' induced by the positive semidefinite matrix ${G_\tau}$ as:
 \[ \| u \|_{G_\tau} = \sqrt{u^T {G_\tau} u} \quad \forall u \in \rset^{N}. \]
Note that $\| u \|_{G_\tau} = 0$ for all $u \in T$ and  $\| u \|_{G_\tau} > 0$ for all $u \in \rset^n \setminus T$. On the subspace $S$ we introduce its extended dual norm:
\begin{align*}
 \| x\|_{G_\tau}^* = \max_{u\in \rset^{N}: \|u \|_{G_\tau} \leq 1} \langle x,
u \rangle  \quad \forall x \in S.
\end{align*}

\noindent Using the definition of conjugate norms, the
Cauchy-Schwartz inequality holds:
\[  \langle u, x \rangle \leq \| u \|_{G_2} \cdot \|
x \|_{G_2}^* \quad \forall x \in S, \;  u \in \rset^{N}.  \]

\noindent Let us define the average value: $\hat u =
\frac{1}{N}\sum_{i=1}^N u_i$.  Then, the dual norm can be computed
for any $x \in S$ as follows:
\begin{align*}
& \| x\|_{G_\tau}^*  = \max_{u \in \rset^{N}: \; \langle {G_\tau} u, u \rangle \leq 1} \langle x, u \rangle = \max_{u: \langle {G_\tau} \left( u -  \hat u e \right), u - \hat u e \rangle \leq 1} \langle x, u - \hat u e \rangle  \\
& = \max_{u: \langle {G_\tau} u, u \rangle \leq 1, \sum_{i=1}^N u_i =0}
\langle x, u \rangle = \max_{u: \langle {G_\tau} u, u \rangle \leq 1, e^Tu=0 } \langle x,
u \rangle \\
&  = \max_{u: \langle {G_\tau} u, u \rangle \leq 1,
u^T e e^T u \leq 0
} \langle x, u \rangle \\
& = \min_{\nu, \mu \geq 0} \max_{u \in \rset^N}  [\langle x, u \rangle + \mu(1 -
\langle {G_\tau} u, u \rangle) - \nu \langle  e e^T u, u \rangle]  \\
& =  \min_{\nu, \mu \geq 0} \mu + \langle (\mu {G_\tau} + \nu e e^T)^{-1} x, x \rangle  = \min_{\nu \geq 0} \min_{\mu \geq 0} [\mu + \frac{1}{\mu} \langle
({G_\tau} + \frac{\nu}{\mu} e e^T)^{-1} x, x \rangle]  \\
& = \min_{\zeta \geq 0} \sqrt{\langle ( {G_\tau} + \zeta e e^T)^{-1} x,
x \rangle}.
\end{align*}

\noindent In conclusion, we obtain an extended  dual norm  that is
well defined  on  subspace~$S$:
\begin{align}
\label{primalnorm}
 \| x \|_{G_\tau}^* = \min_{\zeta \geq 0} \sqrt{\langle \left(
{G_\tau} + \zeta e e^T \right)^{-1}  x, x \rangle} \quad
\forall x \in S.
\end{align}

\noindent Using the eigenvalue decomposition  of the positive
semidefinite matrix ${G_\tau} = \Xi \text{diag}(0, \lambda_2,
\cdots, \lambda_N) \Xi^T$, where $\lambda_i$ are its positive
eigenvalues and $\Xi=[e \; \xi_2 \cdots \xi_N]$ such that $\langle
e, \xi_i \rangle = 0$ for all $i$, then:
\[ ({G_\tau} + \zeta e e^T)^{-1} = \Xi \text{diag}(\zeta
\|e\|^2, \lambda_2, \cdots, \lambda_N)^{-1} \Xi^T. \] From
\eqref{primalnorm} it follows immediately that our defined norm has
the following closed form expression:
\begin{align}
\label{norm_pseudo} \| x \|_{G_\tau}^* = \sqrt{x^T G_\tau^+  x}
\qquad \forall x \in S,
\end{align} where $G_\tau^+ =  \Xi
\text{diag}(0, \frac{1}{\lambda_2}, \cdots, \frac{1}{\lambda_N})
\Xi^T$ denotes the pseudoinverse of the matrix $G_\tau$.


\subsection{Convergence rate: smooth case}
\noindent In order to estimate the rate of convergence of our
algorithm in the smooth case (Assumption \ref{ass}) we introduce the
following distance that takes into account that our algorithm is a
descent method:
\[ {\cal R}(x^0) = \max_{\{x \in S: f(x) \leq f(x^0)\}} \; \min_{x^* \in X^*} \| x - x^*\|_{G_\tau}^*, \]
which measures the size of the level set of $f$ given by $x^0$. We
assume that this distance is finite for the initial iterate $x^0$.
After $k$ iterations of the  algorithm, we generate a random output
$(x^k, f(x^k))$, which depends on the observed implementation of
random variable:
\[ \eta^k =  ({\cal N}_0, \cdots,  {\cal N}_k). \]
Let us define the expected value of the objective function w.r.t. $\eta^k$:
\[  \phi_k = E \left[ f(x^k) \right]. \]

\noindent We now  prove the main result of this section, i.e.
sublinear convergence in mean for the smooth convex case:
\begin{theorem}
\label{th1} Let Assumption  \ref{ass} hold for the optimization
problem \eqref{scpp} and the sequence $(x^k)_{k \geq 0}$ be
generated by algorithm $RCD_\tau$. Then, we have the following
sublinear rate of convergence for the expected values of the
objective function:
\begin{equation}
\label{estimate}
  \phi_k - f^* \le \frac{2 {\cal R}^2(x^0)}{k}.
\end{equation}
\end{theorem}

\begin{proof}
Recall that all our iterates are feasible, i.e. $x^k \in S$.  From
convexity of $f$ and the definition of the norm $\| \cdot
\|_{G_\tau}$ on the subspace $S$,  we get:
\begin{align*}
 f(x^l) - f^* & \leq \langle \nabla f(x^l), x^l - x^* \rangle  \leq \|x^l - x^*\|_{G_\tau}^*  \|\nabla f(x^l)\|_{G_\tau} \\
&  \leq {\cal R}(x^0) \cdot  \|\nabla f(x^l)\|_{G_\tau} \qquad
\forall l \geq 0.
\end{align*} Combining this
inequality with \eqref{expdecreaseM}, we obtain:
\[  f(x^l) - E \left[ f(x^{l+1}) \;|\; x^l \right]  \geq  \frac{(f(x^l) - f^*)^2}{2
{\cal R}^2(x^0)}, \]
 or equivalently
\begin{align*}
E \left[ f(x^{l+1}) \;|\; x^l \right]  - f^* \le f(x^l)- f^* -
\frac{(f(x^l) - f^*)^2}{2 {\cal R}^2(x^0)}.
\end{align*}
Taking the expectation of both sides of this inequality  in
$\eta_{l-1}$ and   denoting $\Delta_l = \phi_l -f^*$ leads to:
\begin{equation*}
\Delta_{l+1}\le \Delta_l - \frac{\Delta_l^2}{2 {\cal R}^2(x^0)}.
\end{equation*}
Dividing both sides of this inequality with $\Delta_{l}
\Delta_{l+1}$ and taking into account that $\Delta_{l+1} \le
\Delta_l$ (see  \eqref{expdecreaseM}), we obtain:
\begin{equation*}
\frac{1}{\Delta_{l}} \le \frac{1}{\Delta_{l+1}} -
 \frac{1}{2 {\cal R}^2(x^0)} \quad \forall l \geq 0.
\end{equation*}
Adding these inequalities from $l=0, \cdots, k-1$ we get that $0
\leq \frac{1}{\Delta_{0}} \le \frac{1}{\Delta_{k}} - \frac{k}{ 2
{\cal R}^2(x^0)}$ from which we obtain the statement
\eqref{estimate} of the theorem. \qed
\end{proof}

\noindent Theorem  \ref{th1} shows that for smooth convex problem \eqref{scpp}
algorithm $RCD_\tau$ has sublinear rate of convergence in
expectation but with a low complexity per iteration. More
specifically,   the complexity per iteration is ${\cal O} (\tau n_f
+ \tau)$,  where $n_f$ is the maximum  cost of computing the
gradient of each function $f_i$ and $\tau$ is the cost of updating
$x^+$. We  assume that the cost of choosing randomly a $\tau$-tuple
of indices ${\cal N}$ for a given probability distribution is negligible (e.g. for $\tau=2$ the cost is $\ln N$).


\subsection{Convergence rate: strongly convex case}
\noindent Additionally to the assumption of Lipschitz continuous
gradient for each function $f_i$ (see Assumption  \eqref{ass}), we
now assume that the function $f$ is also strongly convex with
respect to the extended norm $\|\cdot\|_{G_\tau}^*$ with convexity
parameter $\sigma_{G_\tau}$ on the subspace $S$. More precisely, the
objective function $f$ satisfies for all $x, y \in S$:
\begin{equation} \label{strongq}
f(x) \ge f(y) + \langle \nabla f(y), x -y \rangle +
\frac{\sigma_{G_\tau}}{2} \left( \| x-y \|_{G_\tau}^* \right)^2.
\end{equation}

\noindent  We now derive linear convergence estimates for
algorithm $RCD_\tau$ under the additional strong convexity assumption:
\begin{theorem}
\label{th3sc} Let Assumption  \eqref{ass} hold and additionally we
assume $f$ to be also $\sigma_{G_\tau}$-strongly convex function with respect
to norm $\|\cdot\|_{G_\tau}^*$ (see \eqref{strongq}). Then, for the
sequence $(x^k)_{k \geq 0}$ generated by algorithm $RCD_\tau$ we
have the following linear estimate for the convergence rate in
expectation:
\begin{equation}
\label{strongconv} \phi_k - f^* \le (1 - \sigma_{G_\tau})^k
\left(f(x^0) - f^* \right).
\end{equation}
\end{theorem}

\begin{proof}
From \eqref{expdecreaseM} we have:
\begin{equation*}
2\left( f(x^k) -  E \left[ f(x^{k+1}) \;|\; x^k \right] \right) \ge
\|\nabla f(x^k)\|_{G_\tau}^{2}.
\end{equation*}
On the other hand, minimizing both sides of inequality
\eqref{strongq} over $x \in S$ we have:
\begin{equation*}
\|\nabla f(y)\|_{G_\tau}^{2} \ge 2 \sigma_{G_\tau} (f(y) - f^*)
\quad \forall y \in  S
\end{equation*}
and for $y = x^k$ we get:
\begin{equation*}
\|\nabla f(x^k)\|_{G_\tau}^{2} \ge 2 \sigma_{G_\tau} \left( f(x^k) -
f^* \right).
\end{equation*}
Combining the first inequality with the last one, and taking
expectation in $\eta_{k-1}$ in both sides, we prove the statement of
the theorem.  \qed
\end{proof}

\noindent  We  notice that if $f_i$'s are strongly convex functions
with respect to the Euclidian norm, with convexity parameter
$\sigma_i$, i.e.:
\begin{equation*}
f_i(x_i) \ge f_i(y_i) + \langle\nabla f_i(y_i), x_i-y_i\rangle +
\frac{\sigma_i}{2} |x_i-y_i|^2 \quad \forall x_i, y_i,
\quad  i \in [N],
\end{equation*}
then  the whole  function $f = \sum_i f_i$ is also strongly convex
w.r.t. the extended norm induced by the positive definite matrix
$D_\sigma$, where $\sigma=[\sigma_1 \cdots \sigma_N]^T$,
i.e.:
\[ f(x) \ge f(y) + \langle\nabla f(y), x-y\rangle +
\frac{1}{2} \|x - y \|^2_{D_\sigma} \quad
\forall x, y.
\]
Note that in this extended norm  $\| \cdot \|_{D_\sigma}$ the strongly convex parameter of the function $f$ is equal to $1$. It follows immediately that the function $f$ is also
strongly convex with respect to the norm  $\|\cdot\|_{G_\tau}^*$ with
the strongly convex parameter $\sigma_{G_\tau}$ satisfying:
\begin{equation*}
\sigma_{G_\tau} D_\sigma^{-1} \preceq {G_\tau} + \zeta e e^T,
\end{equation*}
for some  $\zeta \ge 0$. In conclusion, the strong convexity parameter $\sigma_{G_\tau}$ needs to satisfy the following LMI:
\begin{align}
\label{eq_scparam}
\sigma_{G_\tau} I_N \preceq D_\sigma^{1/2} (G_\tau + \zeta e e^T)D_\sigma^{1/2}.
\end{align}

\noindent Finally,  we should notice that we can also easily derive
results showing that the problem is approximately solved with high
probability  in both situations, smooth and/or  strongly convex
case,  see e.g. \cite{Nec:13} for details.


\subsection{How the number of updated blocks $\tau$ enters into the convergence rates}
\noindent Note that matrix $G_\tau$ depends directly on the number of
components $\tau$ to be updated and therefore, $R(x^0)$ is
also depending on $\tau$. Moreover, the convergence rate can be
explicitly expressed in terms of $\tau$ for some specific choices
for probabilities and for the graph ${\cal G}$. In particular, let
us assume a complete graph ${\cal G}$ and that we know some
constants $R_i>0$ such that for any $x$ satisfying $f(x) \leq
f(x^0)$ there exists an $x^* \in X^*$ such that:
\[ |x_i - x_i^* | \le R_i \quad \forall i \in [N], \] and recall
that $L=[L_1 \cdots L_N]^T$ and $D_L$ is the diagonal matrix with
entries on the diagonal given by the vector $L$. Moreover, let us
consider probabilities depending on the Lipschitz constants $L_i$
for any path ${\cal N} \in {\cal P}_\tau$ of $\tau$ vertices in  the
complete graph ${\cal G}$, defined as:
\begin{align}
\label{probinv} p_{\cal N}^{L} = \frac{\sum_{i \in {\cal N}}
1/L_i}{\sum_{{\cal N} \in {\cal P}_\tau}  \sum_{i \in {\cal N}}
1/L_i }.
\end{align}

\begin{theorem}
Under Assumption \ref{ass} and  for the choice of probabilities
\eqref{probinv} on a complete graph ${\cal G}$, the following
sublinear convergence rate for algorithm $RCD_\tau$ in the expected
values of the objective function is obtained:
\[  \phi_k - f^* \leq  \frac{N-1}{\tau-1} \cdot \frac{2 \sum_i L_i R_i^2}{k}.  \]
\end{theorem}

\begin{proof}
Using the definition of the indicator function $1_{\cal N}$, we can
see that:
\begin{align*}
\Sigma_\tau^{-1} & = \sum_{\cal N \in {\cal P}_\tau} \sum_{i \in {\cal N}}
\frac{1}{L_i} = \sum_{j=1}^{\binom{\tau}{N}} \sum_{i \in {\cal
N}_j} \frac{1}{L_i} = \sum_{j=1}^{\binom{\tau}{N}} \sum_{i=1}^N 1_{{\cal N}_j}(i) \frac{1}{L_i}  \\
& = \sum_{i=1}^N \frac{1}{L_i} \left( \sum_{j=1}^{\binom{\tau}{N}}
1_{{\cal N}_j}(i) \right) =  \sum_{i=1}^N \frac{1}{L_i}
\binom{\tau-1}{N-1}.
\end{align*}
Thus, using  $G_\tau = \sum_{\cal N \in {\cal P}_\tau} p_{\cal N}
G_{\cal N}$ and the expression of $G_{{\cal N}}$  given in
\eqref{exp_GN}, we can  derive that matrix $G_\tau$ has the
following expression:
\begin{align*}
G_{\tau} &  = \frac{1}{\Sigma_\tau^{-1}} \sum_{\cal N \in {\cal P}_\tau} \left[
\sum_{i \in {\cal N}} \frac{1}{L_i}  D_{L_{\cal N}}^{-1} - L^{-1}_{\cal N} (L^{-1}_{\cal N})^T \right] \\
&  = \frac{1}{\Sigma_\tau^{-1}} \sum_{j=1}^{\binom{\tau}{N}} \left[
\sum_{i \in {\cal N}_j} \frac{1}{L_i} D_{L_{{\cal N}_j}}^{-1} - L^{-1}_{{\cal N}_j} (L^{-1}_{{\cal N}_j})^T \right] \\
& = \frac{1}{\Sigma_\tau^{-1}}  \binom{\tau-2}{N-2} \left[
\sum_{i=1}^N \frac{1}{L_i} D_{L}^{-1} - L^{-1} (L^{-1})^T \right] \\
&= \frac{\tau-1}{N-1}   \left[ D_{L}^{-1} - \frac{1}{e^T L^{-1}}
L^{-1} (L^{-1})^T \right],
\end{align*}

\noindent Using the previous expression for $G_\tau$ in
\eqref{norm_pseudo}, we get:
\begin{align}
\label{norminv} (\|x\|_{G_\tau}^*)^2 = \frac{N-1}{\tau-1} \sum_{i
\in [N]} L_i |x_i|^2 \quad \forall x \in S.
\end{align}
Using the definition for $R_i$ and the expression \eqref{norminv}
for the norm on $S$, we obtain:
\[ {\cal R}^2(x^0) \leq  {\cal R}^2(R) =  \frac{N-1}{\tau-1}
\sum_{i \in [N]} L_i R_i^2,  \] where $R=[R_1 \cdots R_N]^T$. Using
this expression for ${\cal R}(x^0)$ in Theorem \ref{th1} we get the
statement of the theorem. \qed
\end{proof}

\noindent Note that for  $\tau=N$ we recover the convergence rate of
the full gradient method, while for $\tau=2$ we get the convergence
rate of the basic random coordinate descent method. Thus, the
theoretical speedup of the parallel algorithm $RCD_\tau$ in terms of
the number of iterations needed to approximately solve the problem,
as compared to the basic random coordinate descent method, is equal
in this case to $\tau$ - the number of components to be updated
(number of processors available). This result also shows that the
speedup achieved by our method on the class of separable problems
with coupling constraints is the same as for separable problems
without coupling constraints.

\noindent For the strongly convex case, we  note that combining the Lipschitz inequality
\eqref{lip_in} with the strong convex inequality \eqref{strongq} we get:
\[  \sum_{i \in [N]} L_i | x_i - y_i|^2 \geq \sigma_{G_\tau}
 \left( \| x-y \|_{G_\tau}^* \right)^2 \quad \forall x, y \in S. \]
Now, if we consider e.g. a complete graph and the probabilities
given in \eqref{probinv}, then using the expression for the norm $\|
\cdot \|_{G_\tau}$ given in \eqref{norminv} we obtain
$\sigma_{G_\tau} \leq \frac{\tau -1}{N-1}$. Thus, in
the strongly convex  case the linear convergence
rate in expectation \eqref{strongconv} can be also expressed in terms of $\tau$.


\section{Design of optimal probabilities}
\label{sec_opt_prob}
\label{subsecprob} We have several choices for the probabilities
$(p_{\cal N})_{{\cal N} \in {\cal P}_\tau}$ corresponding to paths
of  $\tau$ vertices in the complete graph ${\cal G}$, which the
randomized coordinate descent algorithm $RCD_\tau$ depends on.
For example, we can choose probabilities dependent on the Lipschitz
constants $L_i$:
\begin{align}
\label{pl} p_{\cal N}^\alpha = \frac{\sum_{i \in {\cal
N}}L_{i}^\alpha}{\Sigma^\alpha_\tau}, \quad\;\; \Sigma^\alpha_\tau =
\sum_{{\cal N} \in {\cal P}_\tau} \sum_{i \in {\cal N}}L_{i}^\alpha,
\;\;\; \alpha \in \rset.
\end{align}
Note that for $\alpha = 0$ we recover the uniform probabilities.
Finally, we can design optimal probabilities from the convergence
rate of the method. From the definition of the constants $R_i$ it
follows that:
\[ {\cal R} (x^0) \leq {\cal R}(R) =  \max_{x \in S, |x_i| \leq R_i}    \|x  \|_{G_\tau}^*,
\quad \text{with} \quad  G_\tau = \sum_{\cal N \in {\cal P}_\tau}
p_{\cal N} G_{\cal N}. \] We have the freedom to choose the matrix
$G_\tau$ that depends linearly on the probabilities  $p_{\cal N}$.
For the probabilities $(p_{\cal N})_{{\cal N} \in {\cal P}_\tau}$
corresponding to paths of  $\tau$ vertices in the complete graph
${\cal G}$ we define the following set of matrices:
\[ {\cal M} = \left\{ G_\tau \! : \quad
 G_\tau = \! \sum_{\cal N \in {\cal P}_\tau}
p_{\cal N} G_{\cal N}, \quad  G_{{\cal N}}=  D_{L_{\cal N}}^{-1} -
\frac{1}{\sum_{i \in {\cal N}} 1/L_i}
   L^{-1}_{\cal N} (L^{-1}_{\cal N})^T \right \}.  \]

\noindent Therefore, we search for the probabilities $p_{\cal N}$
that are the optimal solution of the following optimization problem:
\begin{align*}
{\cal R}^*(x^0) = \min_{p_{\cal N}} {\cal R} (x^0) \leq
\min_{p_{\cal N}} {\cal R} (R)  = \min_{G_\tau \in {\cal M}} \max_{x
\in S, |x_i| \leq R_i} \|x\|_{G_\tau}^*.
\end{align*}

\noindent  Let us define $R=[R_1 \cdots R_N]^T$ and $\nu =[\nu_1
\cdots \nu_N]^T$.  In the next theorem we derive an easily computed
upper bound on ${\cal R}^*(x_0)$ and we provide a way to
suboptimally select the probabilities $p_{\cal N}$:
\begin{theorem}
\label{thsdp} Let Assumption \ref{ass} hold.  Then, a suboptimal
choice of probabilities $(p_{\cal N})_{{\cal N} \in {\cal P}_\tau}$
can be obtained as a solution of the   following SDP problem whose
optimal value is an upper bound on ${\cal R}^*(x_0)$, i.e.:
\begin{align}
\label{sdpopt} \left( {\cal R}^*(x_0) \right)^2 \leq & \min_{G_\tau
\in {\cal M}, \zeta \geq 0, \nu \geq 0} \left\{ \langle \nu, R^2
\rangle: \;\;
\begin{bmatrix}
G_\tau + \zeta e e^T & I_N \\
I_N & D_{\nu} \
\end{bmatrix}
\succeq 0 \right \}.
\end{align}
\end{theorem}

\begin{proof}
Using the definition of ${\cal R} (R)$ and of the norm
$\|\cdot\|_{G_2}^*$ we get:
\begin{align*}
\min_{p_{\cal N}} ({\cal R}(R))^2 & \leq  \min_{{G_\tau} \in {\cal
M}} \; \max_{x \in S,\|x_i\| \leq R_i} \left( \|x\|_{G_\tau}^* \right)^2  \\
&= \min_{{G_\tau} \in {\cal M}} \; \max_{x \in S, \|x_i\| \leq R_i}
\; \min_{\zeta \geq 0} \langle ( {G_\tau} + \zeta e e^T)^{-1} x, x \rangle  \\
& = \min_{{G_\tau} \in {\cal M}, \zeta \geq 0} \; \max_{x \in S,
\|x_i\| \leq R_i} \langle ({G_\tau}+\zeta e e^T)^{-1} x, x \rangle \\
& =   \min_{{G_\tau} \in {\cal M}, \zeta \geq 0} \; \max_{x \in S,
\|x_i\| \leq R_i} \langle ({G_\tau}+\zeta ee^T)^{-1} , x x^T \rangle\\
&=  \min_{{G_\tau} \in {\cal M}, \zeta \geq 0}  \;  \max_{X \in
{\cal X}} \langle ({G_\tau}+\zeta e e^T)^{-1} , X \rangle,
\end{align*}
where ${\cal X} = \{X: \; X \succeq 0, \;  \text{rank} X=1, \;
\langle e e^T, X \rangle =0, \; \langle X, E_{ii} \rangle \leq R_i^2
\;\; \forall i \}$  and  $E_{ii} = e_i e_i^T$. Using the well-known
relaxation from the SDP literature, we have:
\begin{align*}
\min_{p_{\cal N}} ({\cal R}(R))^2 \leq \min_{{G_\tau} \in {\cal M},
\zeta \geq 0}  \;\;  \max_{X \in {\cal X}_r} \langle ({G_\tau}+\zeta
e e^T)^{-1} , X \rangle,
\end{align*}
where ${\cal X}_r = \{X: X \succeq 0,  \langle e e^T, X \rangle =0,
\;  \langle X, E_{ii} \rangle \leq R_i^2 \;\; \forall i\}$, i.e. we
have removed the rank constraint: $\text{rank} X=1$. Then, the right hand side of the
previous optimization problem can be reformulated equivalently,
using Lagrange multipliers,  as follows:

\begin{align*}
& \hspace{2cm} \min_{{G_\tau} \in {\cal M}, \zeta \geq 0} \; \max_{X
\in {\cal X}_r} \langle (G_\tau + \zeta e e^T)^{-1}, X \rangle  \\
& = \min_{G_\tau \in {\cal M}, \zeta, \nu \geq 0,  Z \succeq 0,
\theta \in \rset} \;\;  \max_{X \in \rset^{N \times N}} \Big[
\langle (G_\tau + \zeta e e^T)^{-1} + Z + \theta e e^T, X \rangle \\
& \hspace{6.9cm} + \sum_{i=1}^N \nu_i (R_i^2 - \langle X, E_{ii}
\rangle) \Big ].
\end{align*}
where $\nu =[\nu_1 \cdots \nu_N]^T$.  Rearranging the terms, we can
write the previous convex problem equivalently:
\begin{align*}
& \min_{{G_\tau} \in {\cal M}, \theta \in \rset, Z \succeq 0, \zeta,
\nu \geq 0} \Big [ \sum_i \nu_i R_i^2 \\
& \hspace{2cm} +   \max_{X \in \rset^{N \times N}} \langle
({G_\tau}+\zeta e e^T)^{-1} + Z + \theta e e^T - \sum_i \nu_i
E_{ii}, X \rangle \Big]   \\
& \hspace{2cm} = \min_{(G_\tau, Z, \zeta,\nu, \theta) \in {\cal F}}
\sum_i \nu_i R_i^2,
\end{align*}
where the feasible set is described as: ${\cal F} =\big\{(G_\tau, Z,
\zeta,\nu, \theta): \; {G_\tau} \in {\cal M},  \theta \in \rset, Z
\succeq 0, \zeta,\nu \geq 0, ({G_\tau}+\zeta ee^T)^{-1} + Z + \theta
e e^T - \sum_i \nu_i E_{ii} = 0 \big\}$. Moreover, since $Z \succeq
0$, the feasible set can be rewritten as:
\[  \big\{ (G_\tau,\zeta,\nu,\theta): \; {G_\tau} \in {\cal M},
\theta \in \rset, \zeta,\nu \geq 0, \sum_i \nu_i E_{ii} -
({G_\tau}+\zeta e e^T)^{-1} - \theta e e^T \succeq 0 \big \}. \] We
observe that we can take $\theta=0$ and then we get the feasible
set:
\[  \big\{ (G_\tau,\zeta,\nu):  \; {G_\tau} \in {\cal M}, \;
\zeta,\nu \geq 0, \;  G_\tau + \zeta e e^T \succeq D_\nu^{-1}
\big\}.
\]

\noindent In conclusion, we obtain the following SDP:
\begin{align*}
\min_{p_{\cal N}} ({\cal R}(R))^2 \leq \min_{G_\tau \in {\cal M},
\zeta, \nu \geq 0, G_\tau + \zeta e e^T \succeq D_\nu^{-1} } \langle
\nu, R^2 \rangle.
\end{align*}
\noindent Finally, the SDP  \eqref{sdpopt} is obtained from   Schur
complement formula applied to the previous optimization problem.\qed
\end{proof}

\noindent  Since we assume a connected graph ${\cal G}$, we have
that $\lambda_1({G_\tau}) = 0$ is simple and consequently
$\lambda_2({G_\tau}) > 0$.  Then, the following equivalence holds:
\begin{align}
\label{equivalence_lambda2} {G_\tau} + t \frac{e e^T}{\|e\|^2}
\succeq t I_N \quad \text{if and only if} \quad t \leq
\lambda_2({G_\tau}),
\end{align}
since the spectrum of the matrix $ {G_\tau} + \zeta e e^T$ is $\{
\zeta \|e\|^2, \lambda_2({G_\tau}), \cdots, \lambda_N({G_\tau}) \}$.
It follows that  $\zeta = \frac{t}{\|e\|^2}$, $\nu_i = \frac{1}{t}$
for all $i$, and ${G_\tau}$ such that $t \leq \lambda_2({G_\tau})$
is feasible  for the SDP problem \eqref{sdpopt}. We conclude that:
\begin{align}
\label{convlambda2} \left({\cal R}^* (x^0) \right)^2 & \leq  \min_{{G_\tau}
\in {\cal M}, \zeta, \nu \geq 0, {G_\tau} + \zeta e e^T \succeq
D_\nu^{-1}} \langle \nu, R^2 \rangle \nonumber  \\
& \leq \min_{{G_\tau} \in {\cal M},  t \leq \lambda_2({G_\tau})}
\sum_{i=1}^N R^2_i \frac{1}{t} \leq  \frac{\sum_i
R^2_i}{\lambda_2({G_\tau})} \quad \forall {G_\tau} \in {\cal M}.
\end{align}

\noindent Then, according to Theorem \ref{th1} we obtain the
following upper bound on the  rate of convergence for the expected
values of the objective function in the smooth convex case:
\begin{equation}
\label{estimate3}
  \phi_k - f^* \le \frac{2 \sum_{i=1}^N R^2_i}{\lambda_2(G_\tau) \cdot k}
  \qquad \forall {G_\tau} \in {\cal M}.
\end{equation}

\noindent From the convergence rate for algorithm $RCD_\tau$ given
in \eqref{estimate3} it follows that we can choose the probabilities
such that we maximize the second eigenvalue of $G_\tau$:
\[ \max_{G_\tau \in {\cal M}} \lambda_2(G_\tau). \]

\noindent In conclusion, in order to find some suboptimal
probabilities $(p_{\cal N})_{{\cal N} \in {\cal P}_\tau}$, we can
solve the following simpler SDP problem than the one given in
\eqref{sdpopt}:

\begin{corollary}
From  \eqref{equivalence_lambda2} we get for the smooth case  that a
suboptimal choice of probabilities $(p_{\cal N})_{{\cal N} \in {\cal
P}_\tau}$ can be obtained as a solution of the   following  SDP
problem:
\begin{align}
\label{sdp2} & p_{\cal N}^*  = \arg \max_{t, G_\tau \in {\cal M}}
  \left \{t: \;\; G_\tau  \succeq t \left( I_N -
\frac{ e e^T}{\|e\|^2} \right) \right \}.
\end{align}
\end{corollary}

\noindent Note that the matrices on both sides of the LMI from
\eqref{sdp2} have the common eigenvalue zero associated to the
eigenvector $e$, so that this LMI has empty interior which can cause
problems for some classes of interior point methods. We can overcome
this problem by replacing the LMI constraint in \eqref{sdp2} with
the following equivalent LMI:
\[  G_\tau + \frac{e e^T}{\|e\|^2} \succeq t \left( I_N -
\frac{e e^T}{\|e\|^2} \right).  \]

\noindent Finally, when the functions $f_i$ are $\sigma_i$-strongly
convex, from Theorem \ref{th3sc} and the LMI \eqref{eq_scparam} it
follows  that in order to get a better convergence rate  we need to
search for $\sigma_{G_\tau}$ as large as possible. Therefore, we get
the following result:

\begin{corollary}
For the strongly convex case  the optimal probabilities are chosen
as the solution of the following SDP problem:
\begin{align}
\label{sdpsc} p_{\cal N}^*  =   \arg & \max_{ \sigma_{G_\tau}, \zeta \geq 0,
G_\tau \in {\cal M}} \left\{ \sigma_{G_\tau} :  \quad \sigma_{G_\tau} I_N \preceq D_\sigma^{1/2} (G_\tau + \zeta e e^T) D_\sigma^{1/2}  \right \}.
\end{align}
\end{corollary}

\noindent In \cite{XiaBoy:06}, the authors propose a (center-free)
distributed  scaled gradient method in the form \eqref{center_free}
to solve the separable optimization problem \eqref{scpp} with
strongly convex objective function, where at each iteration the full
gradient needs to be computed. A similar rate of convergence is
obtained as in Theorem \ref{th3sc} under the Lipschitz and  strong
convexity assumption on $f$, where the weights are designed by
solving an SDP in the form \eqref{sdpsc}. Our randomized algorithm
also belongs to this class of methods and  for $\tau=N$ we recover a
version of the method  in \cite{XiaBoy:06}. Moreover, our
convergence analysis covers  the smooth case, i.e. without the
strong convexity assumption.


\section{Applications}
\label{sec_app} \noindent Problem \eqref{scpp} arises in many real
applications, e.g. image processing \cite{BauBor:96,Com:96,Wri:10},
resource allocation \cite{Nec:13,XiaBoy:06} and coordination in
multi-agent systems \cite{Nec:13,YouXie:11}. For example, we can
interpret \eqref{scpp} as $N$ agents exchanging $n$ goods to
minimize a total cost, where the constraint $\sum_i x_i = 0$ is the
equilibrium or market clearing constraint. In this context $[x_i]_j
\geq  0$ means that agent $i$ receives $[x_i]_j$ of good $j$ from
exchange and $[x_i]_j < 0$ means that agent $i$ contributes
$|(x_i)_j|$ of good $j$ to exchange. It can be also viewed as the
distributed dynamic energy management problem: $N$ devices exchange
power in time periods $t = 1, \cdots, n$. Furthermore,  $x_i \in
\rset^n$  is the power flow profile for device $i$ and $f_i(x_i)$ is
the cost of profile $x_i$ (and usually encodes constraints). In this
application the constraint $\sum_i x_i = 0$ represents the energy
balance (in each time period).

\noindent Problem \eqref{scpp}  can also be seen as the dual  corresponding to an
optimization of a sum of convex functions. Consider the following primal
convex optimization problem that arises in many engineering
applications:
\begin{equation}
\label{dual_scp} g^* =  \min_{v \in \cap_{i=1}^N Q_i }   g_1(v)+
\cdots + g_N(v),
\end{equation}
where $g_i$ are all $\sigma_i$-strongly convex functions and $Q_i \subseteq \rset^n$
are convex sets. Denote with $v^*$ the unique optimal solution of
problem \eqref{dual_scp}. This problem can be reformulated as:
\begin{equation*}
 \min_{u_i \in Q_i, u_i = v \; \forall i \in [N]}   g_1(u_1)+  \cdots +
 g_N(u_N).
\end{equation*}
Let us define $u=[u_1^T  \cdots u_N^T]^T$  and $g(u) = g_1(u_1)+
\cdots + g_N(u_N)$.  By duality, using the Lagrange multipliers
$x_i$ for the constraints $u_i=v$, we obtain the equivalent  convex
problem \eqref{scpp}, where $f_i(x_i) = \tilde g_i^*(x_i)$ and
$\tilde g_i^*$ is the convex conjugate of the function $\tilde g_i =
g_i + 1_{Q_i}$, i.e.
\begin{align}
\label{defdualfi}
f_i(x_i) =
\max_{u_i \in Q_i} \langle x_i, u_i \rangle - g_i(u_i) \qquad
\forall i.
\end{align}
Further we have $f^* + g^*=0$.  Note that if $g_i$ is
$\sigma_i$-strongly convex, then the convex conjugate $f_i$ is
well-defined and has Lipschitz continuous gradient with constants
$L_i = \frac{1}{\sigma_i}$ (see \cite{Nes:04}), so that Assumption \ref{ass} holds. A
particular application is the problem of finding the projection of a
point $v_0$ in the intersection of the convex sets $\cap_{i=1}^N
Q_i \subseteq \rset^n$. This problem can be written as an optimization problem in the
form:
\begin{equation*}
 \min_{v \in \cap_{i=1}^N Q_i }   p_1  \| v-v_0 \|^2 + \cdots + p_N  \| v-v_0 \|^2,
\end{equation*}
where $p_i>0$ such that $\sum_i p_i=1$. This is a particular case of
the separable problem \eqref{dual_scp}. Note that since the
functions $g_i(v) = p_i \| v-v_0 \|^2 $ are strongly convex, then
$f_i$ have Lipschitz continuous gradient with constants
$L_i = 1/p_i$ for all $i \in [N]$.

\noindent  We now show how we can  recover an approximate primal
solution for the primal problem \eqref{dual_scp} by solving the corresponding dual
problem \eqref{scpp} with algorithm $RCD_\tau$. Let us define for
any dual variable $x_i$ the primal variable:
\[  u_i(x_i) = \arg \min_{u_i \in Q_i} g_i(u_i) - \langle x_i, u_i \rangle \quad \forall
i \in [N]. \]
Let us  define $\hat v^* = e \otimes v^*$, where $\otimes$ is the  Kronecker
product,  and the norm $\|u\|^2_{D_\sigma} = \sum_i \sigma_i \|u_i\|^2$. Moreover, let $\sigma_{\min} = \min_i \sigma_i$ and  $\lambda_N$  the largest eigenvalue of $G_\tau$.  Furthermore, for simplicity of the presentation we consider the initial starting point $x^0=0$. Then, we can derive convergence estimates on primal infeasibility and suboptimality for \eqref{dual_scp}.

\begin{theorem}
\label{th_primal} For the convex optimization problem
\eqref{dual_scp} we assume  that all functions $g_i$ are
$\sigma_i$-strongly convex. Let $x^k$ be the sequence generated by
algorithm $RCD_\tau$  for solving the corresponding dual problem
\eqref{scpp} and the primal sequence $u^k = u(x^k)$. Then, we have
the following convergence estimates in the expected values  on
primal infeasibility and suboptimality:
\begin{align*}
& E  \left[ \|u^k - \hat v^*\|_{D_\sigma}^2 \right]  \leq \frac{4 R^2(x^0)}{k} \quad \text{and} \quad E  \left[ |g(u^k) - g^*| \right] \leq   \frac{4 R^2(x^0) \lambda_{N}}{\sigma_{\min} \sqrt{k}}.
\end{align*}
\end{theorem}

\begin{proof}
Since all the functions $g_i$ are $\sigma_i$-strongly convex, then
the objective function $\sum_{i=1}^N g_i(u_i) - \langle x_i ,
u_i\rangle$ is also $1$-strongly convex in the variable $u$ w.r.t.
the norm $\|u\|^2_{D_\sigma} = \sum_i \sigma_i \|u_i\|^2$. Using
this property  and the expression of $u_i(x_i)$, we obtain the
following inequalities:
\begin{align}
\label{strongL}  \frac{1}{2} \|u(x) & -  \hat v^*\|_{D_\sigma}^2  =
\sum_{i=1}^N \frac{\sigma_i}{2}\|u_i(x_i) - v^*\|^2 \nonumber\\
&\leq \left(\sum_{i=1}^N g_i(v^*) - \langle x_i, v^*\rangle \right)  - \left( \sum_{i=1}^N g_i(u_i(x_i)) - \langle x_i , u_i(x_i)\rangle \right) \\
& = (-f^*) - (- f(x)) = f(x) - f^* \quad \forall x \in S. \nonumber
\end{align}
Now, let us consider for $x$ the sequence $x^k$ generated by
algorithm $RCD_\tau$ and let $u^k=u(x^k)$. We note that $u_i^{k+1} = u_i^k$ for all $i \in [N] \setminus {\cal N}_k$ and $u_i^{k+1} = u_i(x_i^{k+1})$ for all $i \in  {\cal N}_k$. Taking expectation over the entire history $\eta^k$  and using Theorem \ref{th1}, we get an estimate on primal infeasibility:
\[ E[\|u^k - \hat v^*\|_{D_\sigma}^2] \!\leq\! 2 E[f(x^k) - f^*]
\!=\! 2 (\phi_k - f^*) \!\leq\! \frac{4 R^2(x^0)}{k}.  \]

\noindent Moreover, for deriving estimates on primal suboptimality, we first observe:
\[  \|u\|_{G_\tau} \leq \frac{\lambda_{N}}{\sigma_{\min}} \|u\|_{D_\sigma} \quad
\forall u, \]
and combining with \eqref{strongL} we get:
\begin{align}
\label{strongLG}
\|u(x)  -  \hat v^*\|_{G_\tau} \leq \frac{\lambda_{N} \sqrt{2(f(x) - f^*)}}{\sigma_{\min}}.
\end{align}

\noindent For the left hand side suboptimality, we proceed as follows:
\begin{align*}
f(x^*) & =  \langle x^*, \hat v^* \rangle - g(\hat v^*)  \\
&  = \max_{u_i \in Q_i}  \langle x^*, u  \rangle - g(u) \geq \langle x^*, u(x) \rangle - g(u(x)),
\end{align*}
which leads to the following relation:
\begin{align}
\label{ineq_optl} g(u&(x)) - g^*  \geq \langle x^*,  u(x) - \hat v^* \rangle  \geq - \|x^*\|_{G_\tau}^* \|\hat v^* - u(x) \|_{G_\tau} \nonumber \\
& \overset{\eqref{strongLG}}{\geq} - \|x^*\|_{G_\tau}^*  \frac{\lambda_N \sqrt{2(f(x) - f^*)}}{\sigma_{\min}} \quad \forall x \in S.
\end{align}
Secondly, from the definition of the dual function, we have:
\[ f(x) = \langle x, u(x) \rangle -  g(u(x)) \quad \forall x \in S. \]
Subtracting $f^* = f(x^*)$ from both sides
and using the complementarity condition $ \langle x^*, u(x^*) \rangle =0$, where $u(x^*) = \hat v^*$,  we get the following relations:
\begin{align*}
g  (u(x))  - g^* & =  \langle x, u(x) \rangle - f(x) + f^* \\
& = f(x^*) - f(x) + \langle x - x^*, u(x^*) \rangle + \langle x, u(x) - u(x^*) \rangle \\
& = f(x^*) + \langle x - x^*, \nabla f(x^*) \rangle - f(x) + \langle x, u(x) - u(x^*) \rangle \\
& \leq \langle x, u(x) - u(x^*) \rangle \leq \| x \|_{G_\tau}^* \|\hat v^* - u(x)\|_{G_\tau} \\
& \leq \left( \|x - x^*\|_{G_\tau}^* + \| x^*\|_{G_\tau}^* \right) \| u(x) - \hat v^* \|_{G_\tau} \\
& \overset{\eqref{strongLG}}{\leq}   \left( \|x - x^*\|_{G_\tau}^* + \| x^*\|_{G_\tau}^* \right) \frac{\lambda_N \sqrt{2(f(x) - f^*)}}{\sigma_{\min}},
\end{align*}
valid for all $x \in S$  and $x^* \in X^*$, where in the first inequality
we used convexity of the function $f$ and  the
relation $\nabla f(x) = u(x)$,  and in the second inequality the Cauchy-Schwartz inequality. Now, using the definition of $R(x^0)$ and that $x^0=0$, replacing $x$ with  the sequence $x^k$ in the previous derivations and taking the expectation over the entire history $\eta^k$, we obtain a bound on primal suboptimality:
\begin{align*}
 E[|g(u^k) - g^*|] & \leq   E \left[ \left( \|x^k - x^*\|_{G_\tau}^* + \| x^*\|_{G_\tau}^* \right) \frac{\lambda_N \sqrt{2(f(x^k) - f^*)}}{\sigma_{\min}} \right] \\
& \leq \frac{2 R(x^0) \lambda_N}{\sigma_{\min}}    \sqrt{2E[f(x^k) - f^*]}  \leq  \frac{4 R^2(x^0) \lambda_N}{\sigma_{\min} \sqrt{k}},
\end{align*}
which gives us a convergence   estimate for primal suboptimality for  problem \eqref{dual_scp}.  \qed
\end{proof}

\noindent In conclusion, the expected values of the distance between
the primal generated points $u_i^k \in Q_i$ and the unique optimal
point $v^* \in \cap_i Q_i$ of \eqref{dual_scp},  i.e. $E[\|u^k_i -
v^*\|]$, is less than ${\cal O}(\frac{1}{\sqrt{k}})$. Similar
convergence rates for other projection algorithms   for solving
convex feasibility problems have been derived in the literature, see
e.g. \cite{BauBor:96, Com:96}.


\section{Numerical experiments}
\label{sec_num_simulations}
\noindent In this section we report some preliminary numerical results on
solving the optimization problem \eqref{scpp}, where the functions $f_i$ are taken as in paper  \cite{XiaBoy:06}:
\begin{align}
\label{apl1}
 f_i(x_i) = \frac{1}{2} a_i(x_i -
c_i)^2 + \log(1 + \exp(b_i(x_i - d_i))) \qquad \forall i \in [N],
\end{align}
where the coefficients $a_i \geq 0, b_i, c_i$ and $d_i$ are
generated randomly with uniform distributions on $[-15, \ 15]$. The
second derivatives of these functions have the following
expressions:
\[ f_i^{''} (x_i) =  a_i + \frac{b_i^2 \exp(b_i(x_i-d_i))}{(1 + \exp(b_i(x_i-d_i)))^2},  \]
which have the following lower and upper bounds:
\[ \sigma_i = a_i \quad \text{and} \quad  L_i = a_i + \frac{1}{4} b_i^2.  \]
We assume that the sum of the variables if fixed to zero, i.e.:
\[\sum_{i=1}^N  x_i=0.  \]
In all our numerical tests we consider a complete graph  and Lipschitz dependent probabilities as in \eqref{probinv} (if not specified otherwise).

\noindent In the first set of experiments, we solve  a single
randomly generated  problem with $N=10^4$ nodes for $\tau  = 2, 4$
and $7$ cores in parallel using MPI. The Fig. 1 displays the
evolution of $f(x^k)- f^*$ along normalized iterations $k/N$ of
algorithm $RCD_\tau$. From the plot we can observe that increasing
the number of cores reduces substantially the number of full
iterations $k/N$.
\begin{figure}[ht]
\label{diferent_tau} \caption{Typical performance of algorithm
$RCD_\tau$ for different numbers of updated variables (processors)
$\tau$: evolution of $f(x^k)- f^*$ along normalized iterations
$k/N$, with $\tau=2, 4$ and 7.}
\centerline{\includegraphics[height=5cm,width=9cm]{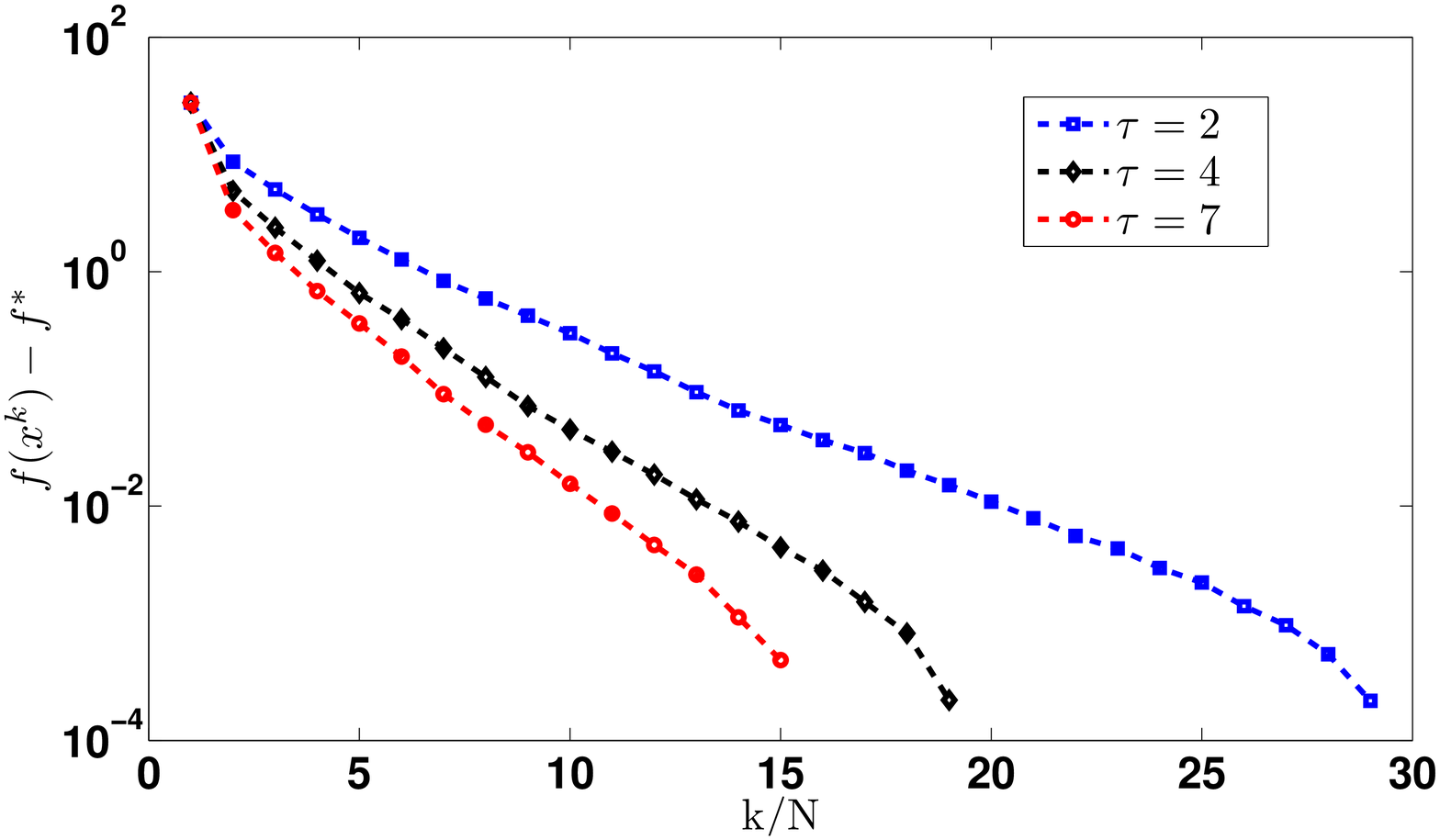}}
\end{figure}

\noindent Then, we tested algorithm $RCD_2$,  i.e. $\tau=2$, and
thus  at each iteration we choose a pair of nodes in the graph
$(i,j) \in E$ with probability $p_{ij}$ and then update only the
components $i$ and $j$ of $x$ as follows:
\[ x^{+}_i \!= x_i + \frac{1}{L_{i} \!+\! L_{j}} \! \left( \nabla f_{j}(x_{j})
\!-\! \nabla f_{i}(x_{i}) \right), \quad x^{+}_j \!= x_j +
\frac{1}{L_{i} \!+\! L_{j}} \! \left( \nabla f_{i}(x_{i}) \!-\!
\nabla f_{j}(x_{j}) \right).
\]
We consider three choices of the probabilities in Fig. 2 for
algorithm $RCD_2$: uniform probability, probabilities depending on
Lipschitz constants as given in \eqref{probinv} and optimal
probabilities obtained from solving the SDP  \eqref{sdp2}. As we
expected, the method based on choosing the optimal probabilities has
the fastest convergence.
\begin{figure}[ht]
\label{different_probab} \caption{Evolution of $f(x^k) - f^*$ along
full iterations for algorithm  $RCD_2$ for  different choices for
the probability: uniform probabilities, Lipschitz dependent
probabilities as given in \eqref{probinv} and optimal probabilities
obtained from solving  the SDP problem \eqref{sdp2}.}
\centerline{\includegraphics[height=5cm,width=9cm]{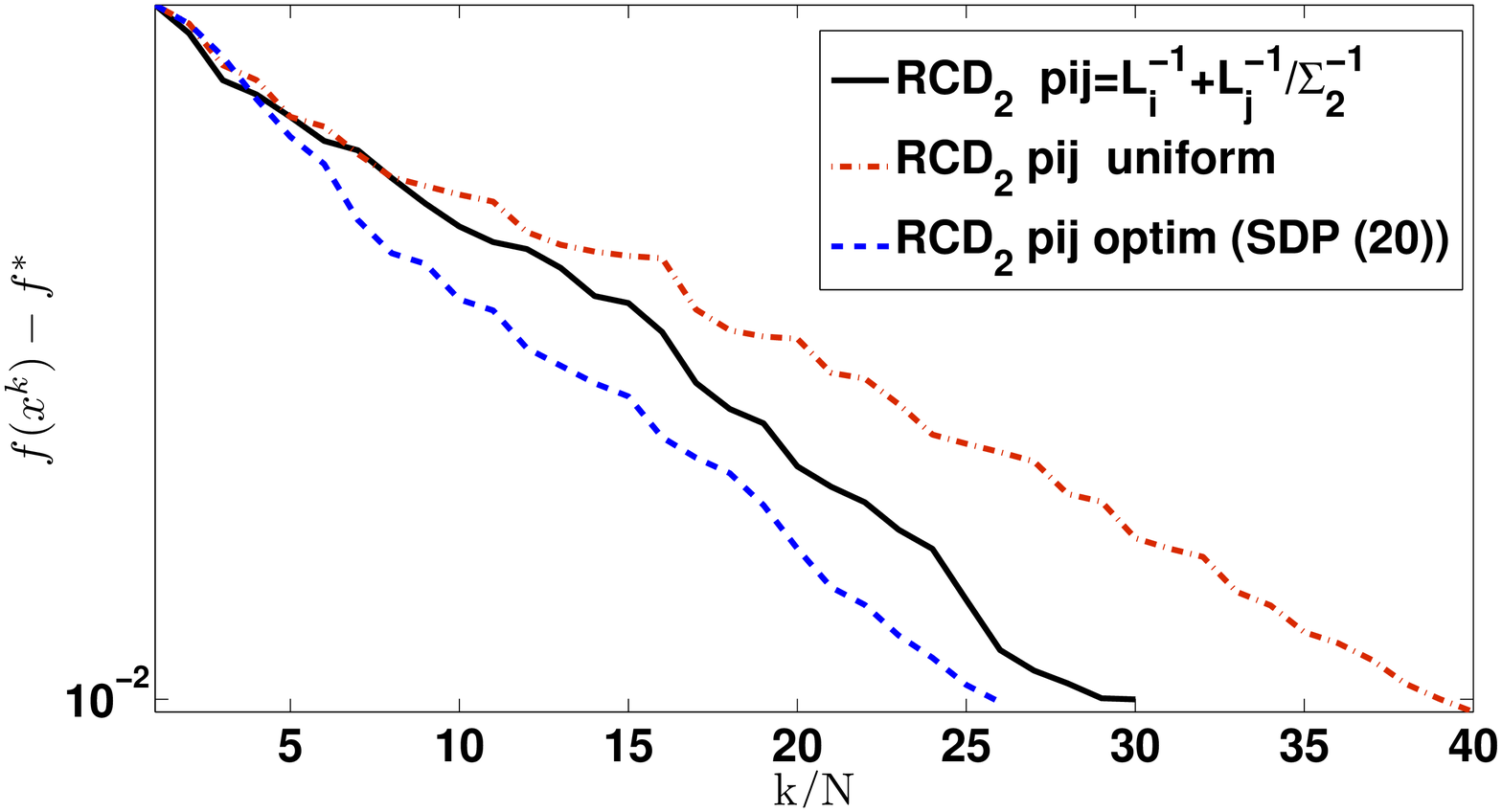}}
\end{figure}

\begin{figure}[ht]
\label{different_methods} \caption{ Evolution of $f(x^k) - f^*$
along full iterations   for the methods:  projected gradient,
center-free gradient  method  with Metropolis weights
\cite{XiaBoy:06} and algorithm  $RCD_\tau$ with uniform and
Lipschitz dependent probabilities for $\tau=2$.}
\centerline{\includegraphics[height=5cm,
width=9cm]{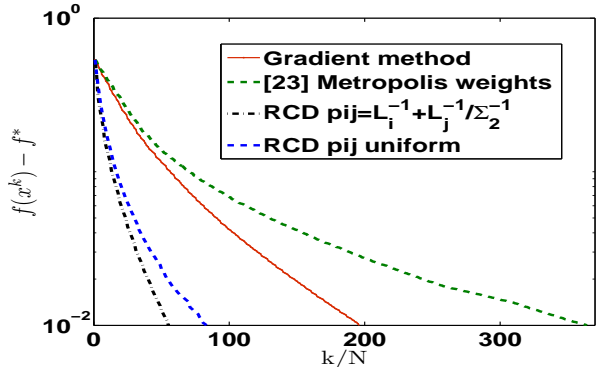}}
\end{figure}

\vspace{0.1cm}

\noindent Finally, we compare our algorithm $RCD_2$, i.e. $\tau=2$,
for two choices for the probabilities $p_{ij}$ (uniform and
Lipschitz dependent probabilities \eqref{probinv}) with the full
gradient method and the center-free gradient method with Metropolis
weights proposed in \cite{XiaBoy:06}. The global Lipschitz constant
in the full gradient is taken as $L_{\max} = \max_i L_i$. Note that
the computations of the local Lipschitz constant $L_i$ required by
algorithm $RCD_\tau$ can be done locally in each node for the
corresponding  function $f_i$, while computing the global Lipschitz
constant $L_{\max}$ for projected gradient method on problems of
very large dimension is difficult. We have also implemented the
center-free gradient  method with Metropolis weights from
\cite{XiaBoy:06}:
\[ x_i^{+} = x_i + \sum_{j \in {\cal N}_i} w_{ij} (\nabla f_j(x_j) -
\nabla f_i(x_i)) \quad \forall i \in [N], \] where ${\cal N}_i$ are
the neighbors of node $i$ in the graph and the weights satisfy the
relation: $\sum_{j \in {\cal N}_i} w_{ij} =0$. In Fig. 3 we plot the
evolution of $f(x^k) - f^*$ along full iterations $k/N$ for the
following methods: the center-free gradient algorithm from
\cite{XiaBoy:06} with the Metropolis weights, the full projected
gradient algorithm and the $RCD_2$ algorithm with uniform and
Lipschitz dependent probabilities. We clearly see that the best
accuracy is achieved by the $RCD_2$ algorithm with Lipschitz
dependent probabilities.


\section{Conclusions}
\noindent In this paper we have derived  parallel random coordinate
descent methods for minimizing linearly constrained convex problems
over networks.  Since we have coupled constraints in the problem, we
have devised an algorithm that updates in parallel $\tau \geq 2$
(block) components per iteration. We have proved that for this
method  we obtain in expectation an $\epsilon$-accurate solution in
at most $\mathcal{O}(\frac{N}{\tau \epsilon})$ iterations and thus
the convergence rate depends linearly on the number of (block)
components to be updated. Preliminary numerical results  show that
the parallel  coordinate  descent method with $\tau>2$  accelerates
on its basic counterpart  corresponding to $\tau=2$. For strongly
convex functions the new method converges linearly. We have also
provided SDP formulations that enable us to  choose the
probabilities in an optimal fashion.

%


\end{document}